\documentclass{amsart}

\makeatletter
 \def\LaTeX{\leavevmode L\raise.42ex
   \hbox{\kern-.3em\size{\sf@size}{0pt}\selectfont A}\kern-.15em\TeX}
\makeatother

\newcommand{\BibTeX}{{\rm B\kern-.05em{\sc
i\kern-.025emb}\kern-.08em\TeX}}

\newtheorem{thm}{Theorem}[section]
\newtheorem{lem}[thm]{Lemma}

\theoremstyle{definition}
\newtheorem{defn}{Definition}

\numberwithin{equation}{section}

\begin{document}

\title[Cubature formulas and Discrete Fourier transform on manifolds ]{Cubature formulas and Discrete Fourier transform  on compact  manifolds}
\maketitle

\begin{center}
\title{\textit{Dedicated to Leon Ehrenpreis}}
\maketitle
\end{center}
\begin{center}

\bigskip

\author{Isaac Z. Pesenson }\footnote{ Department of Mathematics, Temple University,
 Philadelphia,
PA 19122; pesenson@temple.edu. The author was supported in
part by the National Geospatial-Intelligence Agency University
Research Initiative (NURI), grant HM1582-08-1-0019. }

\author{Daryl Geller }\footnote{Department of Mathematics, Stony Brook University, Stony Brook, NY 11794-3651  (12/26/1950-01/27/2011)}
\end{center}

\begin{abstract}
The goal of the paper is to describe essentially optimal cubature formulas on compact Riemannian manifolds which are exact on spaces of band-limited functions.

\end{abstract}



 \section{Introduction}

Daryl Geller and I started to work on this paper, dedicated to the memory of Leon Ehrenpreis, in the Fall of 2010.  Sadly, Daryl Geller passed away suddenly in January of 2011. I will always remember him as a good friend and a wonderful mathematician. 

\bigskip

Analysis on two dimensional surfaces and in particular on the sphere $S^{2}$ found
many applications in computerized tomography, statistics, signal
analysis, seismology, weather prediction, and computer vision.
During last years many problems of  classical harmonic analysis were developed for functions on manifolds and  especially for functions on spheres: splines, interpolation, approximation, different aspects of Fourier analysis, continuous and discrete wavelet transform, quadrature formulas. Our list of references is very far from being complete \cite{ANS}-\cite{DH}, \cite{DNW}-\cite{HNSW}, \cite{HMS}-\cite{WF}.
More references can be found in monographs \cite{F}, \cite{LS}.

The goal of the paper is to describe three types of cubature formulas on general compact Riemannian manifolds which require  \textit{essentially optimal number of  nodes}.  Cubature formulas introduced in section 3 are \textit{exact on subspaces of band-limited functions}. Cubature formulas constructed in section 4  are \textit{exact on spaces of variational splines} and, at the same time, \textit{asymptotically exact on spaces of band-limited functions}.
In section 5 we  prove existence  of cubature formulas with \textit{positive weights} which are exact on spaces of band-limited functions.

In section 7 we prove that on homogeneous compact manifolds the product of two band-limited functions is also band-limited. This result  makes our findings about cubature formulas relevant to Fourier transform on homogeneous compact manifolds and allows \textit{exact} computation of Fourier coefficients of band-limited functions on compact \textit{homogeneous} manifolds.

It worth to  note that all results of the first four sections hold true even for non-compact Riemannian manifolds of bounded geometry. In this case one has  properly define spaces of bandlimited functions on  non-compact manifolds \cite{Pes00}.

Let $\bold M$ be a compact Riemannian manifold and $\mathcal L$ is a differential elliptic operator which is self adjoint in $L_{2}(\bold M)=L_{2}(\bold M, dx)$, where $dx$ is the Riemannian measure. 
The spectrum of this operator, say
$0=\lambda_{0}<\lambda_{1}\leq \lambda_{2}\leq ...$,
is discrete and approaches infinity.  Let
$u_{0}, u_{1}, u_{2}, ...$ be a corresponding
complete system of real-valued orthonormal eigenfunctions, and let
$\textbf{E}_{\omega}(\mathcal{L}),\ \omega>0,$ be the span of all
eigenfunctions of $\mathcal{L}$, whose corresponding eigenvalues
are not greater than $\omega$.   For a function $f\in L_{2}(\bold M)$ its Fourier transform is the set of coefficients $\{c_{j}(f)\}$, which are given by formulas
\begin{equation}
\label{FC}
c_{j}(f)=\int_{\bold M}fu_{j}dx.
\end{equation}
By a discrete Fourier transform we understand a discretization of the above formula. Our goal in this paper is to develop cubature formulas of the form
\begin{equation}\label{cubgen}
\int_{\bold M}f\approx \sum_{x_{k}}f(x_{k})w_{k},
\end{equation}
where $\{x_{k}\}$ is a discrete set of points on $\bold M$ and $\{w_{k}\}$ is a set of weights. When creating such formulas one has to address (among others) the following problems:

1.  to make sure that there exists a relatively large class of functions on which such formulas are exact;

2.  to be able to estimate accuracy of such formulas for general functions;

3. to describe optimal sets of points $\{x_{k}\}$ for which the cubature formulas exist;

4. to provide "constructive" ways for determining optimal sets of points $\{x_{k}\}$;

5. to provide "constructive" ways of determining weights $w_{k}$;

6. to describe properties of appropriate weights.

In the first five sections of the paper we construct cubature formulas on general compact Riemannian manifolds and general elliptic second order differential operators. Namely, we have two types of cubature formulas: formulas which are \textit{exact} on spaces $\textbf{E}_{\omega}(\mathcal{L})$ (see section 3), i. e. 
\begin{equation}\label{ER}
\int_{\bold M}f= \sum_{x_{k}}f(x_{k})w_{k}
\end{equation}
and formulas which are \textit{exact} on spaces of variational splines (see section 4). Moreover, the cubature formulas in section 4 are also asymptotically exact on the spaces $\textbf{E}_{\omega}(\mathcal{L}).$ For both types of formulas we address first five issues from the list above. However, in the first four sections we don't discuss the issue 6 from the same list. 

In section 5 we construct another set of cubature formulas which are exact on spaces $\textbf{E}_{\omega}(\mathcal{L})$ which have \textit{positive weights of  the  "right" size}. Unfortunately, for this set of cubatures we unable to provide constructive ways of determining weights  $w_{k}$.

If one considers integrals of the form (\ref{FC}) then  in the general case  we do not have any criterion to determine whether  the product $fu_{j}$ belongs to the space  $\textbf{E}_{\omega}(\mathcal{L})$ in order to have an \textit{exact} relation 
\begin{equation}
\int_{\bold M}fu_{j}= \sum_{x_{k}}f(x_{k})u_{j}(x_{k})w_{k},
\end{equation}
for cubature rules described in sections 1-4.
However, if $\bold M$ is a compact homogeneous manifolds i. e.  $\bold M=G/K$, where $G$ is a compact Lie group and $K$ is its closed subgroup and $\mathcal{L}$ is the second order Casimir operator (see  (\ref{Laplacian}) below) then we can show that for $f, g \in \textbf{E}_{\omega}(\mathcal{L})$ their product $fg$ is in $\textbf{E}_{4d\omega}(\mathcal{L})$, where $d=dim\>G$ (see section 7).

\section{Plancherel-Polya-type inequalities}

Let $B(x,r)$ be a metric ball on ${\bf M}$ whose center is $x$ and
radius is $r$. The following important lemma can be found  in  \cite{Pes00},
\cite{Pes04a}.

\begin{lem}
There exists
a natural number $N_{{\bf M}}$, such that  for any sufficiently small $\rho>0$,
there exists a set of points $\{y_{\nu}\}$ such that:
\begin{enumerate}
\item the balls $B(y_{\nu}, \rho/4)$ are disjoint,

\item  the balls $B(y_{\nu}, \rho/2)$ form a cover of ${\bf M}$,

\item  the multiplicity of the cover by balls $B(y_{\nu}, \rho)$
is not greater than $N_{{\bf M}}.$
\end{enumerate}\label{covL}
\end{lem}

\begin{defn}
Any set of points $M_{\rho}=\{y_{\nu}\}$ which is as described in
Lemma \ref{covL} will be called a metric
$\rho$-lattice.\label{D1}
\end{defn}

To define Sobolev spaces,  we fix a  cover $B=\{B(y_{\nu}, r_{0})\}$ of $\bold{M}$ of finite
multiplicity $N_{\bf M}$ (see Lemma \ref{covL})
\begin{equation}
\bold{M}=\bigcup B(y_{\nu}, r_{0}),\label{cover}
\end{equation}
where $B(y_{\nu}, r_{0})$ is a  ball centered at $y_{\nu}\in \bold{M}$ of radius
$r_{0}\leq \rho_{\bold{M}},$ contained in a coordinate chart, and consider a fixed partition of unity
$\Psi=\{\psi_{\nu}\}$ subordinate to this cover. The Sobolev
spaces $H^{s}(\bold{M}), s\in \mathbf{R},$ are
introduced as the completion of $C^{\infty}(\bold{M})$ with respect
to the norm
\begin{equation}
\|f\|_{H^{s}(\bold{M})}=\left(\sum_{\nu}\|\psi_{\nu}f\|^{2}
_{H^{s}(B(y_{\nu}, r_{0}))}\right) ^{1/2}.\label{Sobnorm}
\end{equation}
Any two such norms are equivalent.  Note that  spaces $H^{s}(\bold{M}), s\in \mathbf{R},$ are domains of 
operators $A^{s/2}$ for all elliptic differential operators $A$ of order $2$. It implies, that for any $s\in \mathbf{R}$ there exist positive constants $a(s),\>b(s)$ (which depend on $\Psi$, $A$) such that
\begin{equation}\label{SnormE}
\|f\|_{H^{s}(\bold{M})}\leq a(s)\left(\|f\|^{2}_{L_{2}(\bold{M})}+\|A^{s/2}f\|_{L_{2}(\bold{M})}\right)^{1/2}\leq b(s)\|f\|_{H^{s}(\bold{M})}
\end{equation}
for all $f\in H^{s}(\bold{M}).$

 We are going to keep notations from the introduction. Since the
operator $\mathcal{L}$ is of order two, the dimension
$\mathcal{N}_{\omega}$ of the space ${\mathbf E}_{\omega}(\mathcal{L})$ is
given asymptotically by Weyl's formula 
\begin{equation}
\mathcal{N}_{\omega}({\bf M})\asymp C({\bf M})\omega^{n/2},\label{W}
\end{equation}
where $n=dim {\bf M}$.

 The next two theorems were proved in \cite{Pes00}, \cite{Pes04b},
for a Laplace-Beltrami operator in $L_{2}(\bf {M})$ on a Riemannian manifold $\bf {M}$ of
bounded geometry,  but their proofs go through for any
elliptic second-order differential operator in  $L_{2}({\bf M})$. 
In what follows the notation $n=dim\  {\bf M}$ is used.

\begin{thm}  There exist constants $\>\>C_{1}>0$ and
 $\>\>\>
\rho_{0}>0,$ such that for any  natural  $m>n/2$,
any $0<\rho<\rho_{0}$,  and any $\rho$-lattice
$M_{\rho}=\{x_{k}\}$,  the following inequality holds:
\begin{equation}
\label{Sob11}
\left(\sum _{x_{k}\in M_{\rho}}|f(x_{k})|^{2}\right)^{1/2}\leq
 C_{1}\rho^{-n/2}\|f\|_{H^{m}({\bf M})},
 \end{equation}
for all  $f\in H^{m}({\bf M}). $\label{T1}
\end{thm}

\begin{thm} There exist constants
$C_{2}>0,$ and $\>\>\>
\rho_{0}>0,$ such that for any  natural $m>n/2$,
any $0<\rho<\rho_{0}$, and any $\rho$-lattice
$M_{\rho}=\{x_{k}\}$ the following  inequality holds

\begin{equation}\label{Sob22}
\|f\|_{H^{m}({\bf M})}\leq C_{2}\left\{\rho^{n/2}\left(\sum_{x_{k}\in M_{\rho}}
|f(x_{k})|^{2}\right)^{1/2}+\rho^{m}\|\mathcal{L}^{m/2}f\|_{L_{2}(\bf{M})}\right\}.
\end{equation}

\end{thm}

 As one can easily verify the norm of $\mathcal{L}$ on the subspace  ${\mathbf E}_{\omega}(\mathcal{L})$ (the span of  eigenfunctions whose eigenvalues $\leq \omega$) is exactly $\omega$.  In particular one has the following Bernstein-type inequality
 \begin{equation}
 \label{BI}
 \|\mathcal{L}^{s}f\|_{L_{2}({\bf M})}\leq \omega^{s}\|f\|_{L_{2}({\bf M})},\>\>\>s\in {\bf R}_{+},
 \end{equation}
 for all $f\in {\mathbf E}_{\omega}(\mathcal{L})$.
 This fact and the previous two theorems imply the
following Plancherel-Polya-type inequalities.  Such
inequalities are also known as Marcinkewicz-Zygmund inequalities.

\bigskip

\begin{thm}
\label{PP2}
Set $m_{0}=\left[\frac{n}{2}\right]+1$. If $C_{1},\>\>C_{2}$ are the same as above, $a(m_{0})$ is from (\ref{SnormE}),  and $c_{0}=\left(\frac{1}{2}C_{2}^{-1}\right)^{1/m_{0}}$
then  for any $\omega>0$, and for
every metric $\rho$-lattice $M_{\rho}=\{x_{k}\}$ with $\rho=
c_{0}\omega^{-1/2}$, the following Plancherel-Polya inequalities
hold:

\begin{equation}
C_{1}^{-1}a(m_{0})^{-1}(1+\omega)^{-m_{0}/2}\left(\sum_{k}|f(x_{k})|^{2}\right)^{1/2}\leq\rho^{-n/2}\|f\|_{L_{2}({\bf M})}
\leq 
$$
$$
(2C_{2})\left(\sum_{k} |f(x_{k})|^{2}\right)^{1/2}, \label{completePlPo100}
\end{equation}
for all $f\in {\mathbf E}_{\omega}(\mathcal{L})$ and $n=\dim \  {\bf M}$. 
\label{completePlPo2}

\end{thm}
\begin{proof}

Since $\mathcal{L}$ is an  elliptic second-order differential operator on a compact manifold which is self-adjoint and positive definite in $L_{2}(\bf{M})$
 the norm on the Sobolev space $H^{m_{0}}(\bf{M})$ is equivalent to the norm $\|f\|_{L_{2}(\bold{M})}+\|\mathcal{L}^{m_{0}/2}f\|_{L_{2}(\bold{M})}$. Thus, the inequality (\ref{Sob11}) implies
$$
\left(\sum _{x_{k}\in {\bf M}_{\rho}}|f(x_{k})|^{2}\right)^{1/2}\leq
 C_{1}a(m_{0})\rho^{-n/2}\left( \|f\|_{L_{2}(\bold{M})}+\|\mathcal{L}^{m_{0}/2}f\|_{L_{2}(\bold{M})} \right).
 $$

The  Bernstein inequality shows that   for all $f\in {\mathbf E}_{\omega}(\mathcal{L})$ and all $\omega\geq 0$
  $$
  \|f\|_{L_{2}(\bold{M})}+\|\mathcal{L}^{m_{0}/2}f\|_{L_{2}(\bold{M})}\leq (1+\omega)^{m_{0}/2}\|f\|_{L_{2}(\bold{M})}.
 $$
 Thus we proved the inequality 
 \begin{equation}
C_{1}^{-1}a(m_{0})^{-1}(1+\omega)^{-m_{0}/2}\left(\sum _{x_{k}\in M_{\rho}}|f(x_{k})|^{2}\right)^{1/2}\leq
 \rho^{-n/2}\|f\|_{L_{2}(\bold{M})}, \>\>\\\ f\in {\mathbf E}_{\omega}(\mathcal{L}).
 \end{equation}

 To prove the opposite inequality we start with inequality (\ref{Sob22}) where $    m_{0}=\left[\frac{n}{2}\right]+1           $. 
Applying the Bernstein inequality  (\ref{BI}) we obtain
\begin{equation}
\label{x}
\|f\|_{L_{2}(\bold{M})}\leq C_{2}\rho^{n/2}\left(\sum_{x_{k}\in M_{\rho}}
|f(x_{k})|^{2}\right)^{1/2}+C_{2}\rho^{m_{0}}\omega^{m_{0}/2}\|f\|_{L_{2}(\bold{M})},
\end{equation}
where  $f\in {\mathbf E}_{\omega}(\mathcal{L})$. Now we fix the following value for $\rho$
$$
\rho=\left(\frac{1}{2}C_{2}^{-1}\right)^{1/m_{0}}\omega^{-1/2}=c_{0}\omega ^{-1/2},\>\>\>c_{0}=\left(\frac{1}{2}C_{2}^{-1}\right)^{1/m_{0}}.
$$
With such $\rho$ the factor in the front of the last term in (\ref{x}) is exactly $1/2$. Thus, this term can be moved to the left side of the formula (\ref{x}) to obtain
\begin{equation}
\label{xx}
\frac{1}{2}\|f\|_{L_{2}(\bold{M})}\leq C_{2}\rho^{n/2}\left(\sum_{x_{k}\in M_{\rho}}
|f(x_{k})|^{2}\right)^{1/2}.
\end{equation}
In other words, we obtain the inequality
$$
\rho^{-n/2}\|f\|_{L_{2}(\bold{M})}\leq 2C_{2}\left(\sum_{x_{k}\in  M_{\rho}}
|f(x_{k})|^{2}\right)^{1/2}.
$$

The theorem is proved.

\end{proof}

It is interesting to note that our $\rho$-lattices (appearing in the
previous Theorems) always produce sampling sets with essentially 
optimal number of sampling points. In other words, the number of points in a sampling set for  ${\mathbf E}_{\omega}(\mathcal{L})$ is "almost" the same as the dimension of the space   ${\mathbf E}_{\omega}(\mathcal{L})$ which is given by the Weyl's formula (\ref{W}).

\begin{thm}
 If the  constant $c_{0}>0$ is the same
as above, then  for any $\omega>0$ and $\rho=c_{0}\omega^{-1/2}$,
there exist positive $a_{1}, \>a_{2}$ such
that the number of points in  any $\rho$-lattice $M_{\rho}$
satisfies the following inequalities
\begin{equation}
a_{1}\omega^{n/2}\leq |M_{\rho}|\leq
a_{2}\omega^{n/2};\label{rate}
\end{equation}

\end{thm}\label{FT}

\begin{proof}
According to the definition of a
lattice $M_{\rho}$ we have
$$
|M_{\rho}|\inf_{x\in M}Vol(B(x,\rho/4))\leq Vol({\bf M})\leq |M_{\rho}|\sup_{x\in
M}Vol(B(x,\rho/2))
$$
or
$$
\frac{Vol(\bold M)}{\sup_{x\in \bold M}Vol(B(x,\rho/2))}\leq |M_{\rho}|\leq
\frac{Vol(\bold M)}{\inf_{x\in \bold M}Vol\left(B(x,\rho/4)\right)}.
$$

Since for  certain $c_{1}(\bold M), c_{2}(\bold M)$, all $x\in \bold M$ and all
sufficiently small $\rho>0$ one has a double inequality
$$
c_{1}(\bold M)\rho^{n}\leq  Vol(B(x,\rho))\leq c_{2}(\bold M)\rho^{n},
$$
and since $\rho=c_{0}\omega^{-1/2}, $ we obtain the inequalities (\ref{rate}) for certain $a_{1}=a_{1}(\bold M),\>\>a_{2}=a_{2}(\bold M).$

\end{proof}

\section{Cubature formulas on manifolds which are exact on band-limited functions}

\bigskip

Theorem \ref{PP2}
 shows that  if $x_{k}$ is in a $\rho$ lattice $M_{\rho}$ and  $\vartheta_{k}$ is
the orthogonal projection of the Dirac measure $\delta_{x_{k}}$ on
the space ${\mathbf E}_{\omega}(\mathcal{L})$ (in a Hilbert space
$H^{-n/2-\varepsilon}({\bf M}), \>\>\varepsilon >0)$ then there exist
constants $c_{1}=c_{1}({\bf M},\mathcal{L}, \omega)>0,\>\>
c_{2}=c_{2}({\bf M},\mathcal{L})>0,$ such that the following frame
inequality holds for all $f\in {\mathbf E}_{\omega}(\mathcal{L})$
\begin{equation}
c_{1}\left(\sum_{k}\left|\left<f,\vartheta_{k}\right>\right|^{2}\right)^{1/2}
\leq \rho^{-n/2}\|f\|_{L_{2}({\bf M})} \leq
c_{2}\left(\sum_{k}\left|\left<f,\vartheta_{k}\right>\right|^{2}\right)^{1/2},
\end{equation}
where 
$$
\left<f,\vartheta_{k}\right>=f(x_{k}),\>\>\>f\in {\mathbf E}_{\omega}(\mathcal{L}).
$$

From here by using the
classical ideas of Duffin and Schaeffer about dual frames
\cite{DS} we obtain the following reconstruction formula.
\begin{thm}
If $M_{\rho}$ is a $\rho$-lattice in Theorem  \ref{PP2} with $\rho=c_{0}\omega^{-1/2}$  then there exists a frame
$\{\Theta_{j}\}$ in the space ${\mathbf E}_{\omega}(\mathcal{L})$ such that
the following reconstruction formula holds for all functions in ${\mathbf E}_{\omega}(\mathcal{L})$
\begin{equation}
f=\sum_{x_{k}\in M_{\rho}}f(x_{k})\Theta_{k}.\label{exactrecon}
\end{equation}\label{Frecon}
\label{4.2}
\end{thm}

This formula implies that for any linear functional $F$ on the space ${\mathbf E}_{\omega}(\mathcal{L})$ one has
$$
F(f)=\sum_{x_{k}\in M_{\rho}}f(x_{k})F(\Theta_{k}), \>\>\>f\in {\mathbf E}_{\omega}(\mathcal{L}).
$$
In particular, we have the following exact cubature formula.
\begin{thm}
If $M_{\rho}$ is a $\rho$-lattice in Theorem  \ref{PP2} with $\rho=c_{0}\omega^{-1/2}$  and 
$$
\nu_{k}=\int_{\bold M}\Theta_{k},
$$
then for all $f\in  {\mathbf E}_{\omega}(\mathcal{L})$ the following holds 
\begin{equation}\label{cub-1}
\int_{\bold M}f=\sum_{x_{k}\in M_{\rho}}f(x_{k})\nu_{k}, \>\>\>f\in {\mathbf E}_{\omega}(\mathcal{L}).
\end{equation}
\end{thm}

Thus, we have a cubature formula which is exact on the space ${\mathbf E}_{\omega}(\mathcal{L})$. Now,we are going to consider general functions $f\in L_{2}(\bold M)$.  Let $f_{\omega}$ be orthogonal projection of $f$ onto space ${\mathbf E}_{\omega}(\mathcal{L})$. As it was shown in \cite{Pes09} there exts a constant $C_{k,m}$ that the following estimate holds for all $f\in L_{2}(\bold M)$
\begin{equation}\label{mod-1}
\|f-f_{\omega}\|_{L_{2}(\bold M)}\leq \frac{C_{
k, m}}{\omega^{k}}\Omega_{m-k}\left(\mathcal L^{k}f, 1/\omega\right),\>\>\>k, m\in \bold N.
\end{equation}
Here the modulus of continuity is defined as
\begin{equation}
\Omega_{r}(g,s)=\sup_{|\tau|\leq
s}\left\|\Delta^{r}_{\tau}g\right\|,\>\>\>g\in L_{2}(\bold M), \>\>\>r\in \bold N,\label{ModCont}
\end{equation}
where
\begin{equation}
\Delta^{r}_{\tau}g=(-1)^{r+1}\sum^{r}_{j=0}(-1)^{j-1}C^{j}_{r}e^{j\tau(i\mathcal L)}g,\>\>\>
\tau\in \mathbf{R}, \>\>\>r\in \mathbf N.\label{Dif}
\end{equation}
Thus, by combining (\ref{cub-1}) and (\ref{mod-1}) we obtain the following theorem.

\begin{thm} There exists a $c_{0}=c_{0}(\bold M,\mathcal L)$  and for
   any $0\leq k\leq m, k,m\in \mathbb{N},$ there exists a constant $C_{k,m}>0$ such
   that if $M_{\rho}=\{x_{k}\}$ is a $\rho$-lattice  with
   $0<\rho\leq c_{0}\omega^{-1}$ then for the same weights $\{\nu_{j}\}$ as in (\ref{cub-1})
\begin{equation}
\left |\int_{\bold M}f-
\sum_{x_{j}}f_{\omega}(x_{j})\nu_{j}\right |\leq\frac{C_{
k, m}}{\omega^{k}}\Omega_{m-k}\left(\mathcal{L}^{k}f, 1/\omega\right),
\end{equation}
where $f_{\omega}$ is the orthogonal projection of $f\in L_{2}(\bold M)$
onto ${\mathbf E}_{\omega}(\mathcal{L})$.
\end{thm}

Note (see \cite{Pes09}), that  $f\in L_{2}(\bold M)$ belongs to the Besov space $\mathbf{B}_{2,\infty}^{\alpha}(\bold M)$ if and only if 
$$
\Omega_{m}\left(f, 1/\omega\right)=O(\omega^{-\alpha}),
$$ when $\omega\longrightarrow\infty$. Thus, we obtain that for functions in $
\mathbf{B}_{2,\infty}^{\alpha}(\bold M)$ the following relation holds
\begin{equation}
\left |\int_{\bold M}f-\sum_{x_{j}}\nu_{j}f_{\omega}(x_{j})\right |=O(\omega^{-\alpha}),\>\>\>\omega\longrightarrow
\infty.
\end{equation}

\section{Cubature formulas on compact manifolds which are exact on variational splines}

\bigskip

Given a $\rho$ lattice $M_{\rho}=\{x_{\gamma}\}$ and a sequence $\{z_{\gamma}\}\in l_{2}$ we
will be
 interested to find a
 function $s_{k}\in H^{2k}(\bold M),$ where $k $ is large enough, such that
\begin{enumerate}
\item $ s_{k}(x_{\gamma})=z_{\gamma}, x_{\gamma}\in \ M_{\rho};$

\item function $s_{k}$ minimizes functional $g\rightarrow \|\mathcal{L}^{k}g\|_{L_{2}(\bold M)}$.
\end{enumerate}
We already know (\ref{Sob11}), (\ref{Sob22})
that for $k\geq d$ the norm $H^{2k}(\bold M)$ is equivalent to the  norm

$$
C_{1}(\rho)\|f\|_{H^{2k}(\bold M)}\leq \|\mathcal{L}^{k}f\|_{L_{2}(\bold M)}+\left (\sum_{x_{\gamma}\in
M_{\rho}}|f(x_{\gamma})|^{2}\right)^{1/2}\leq C_{2}(\rho)\|f\|_{H^{2k}(\bold M)}.
$$
For the given sequence  \  $\{z_{\gamma} \}\in l_{2}$ consider a function $f$
from $H^{2k}(\bold M)$ such that $f(x_{\gamma})=z_{\gamma}.$ Let $Pf$
 denote the orthogonal projection of this function $f$  in the Hilbert
space $H^{2k}(\bold M)$ with the  inner product
$$<f,g>=\sum_{x_{\gamma}\in
M_{\rho}}f(x_{\gamma})g(x_{\gamma})+ <\mathcal{L}^{k/2}f,\mathcal{L}^{k/2}g>$$
on the subspace
$U^{2k}(M_{\rho})=\left \{f\in H^{2k}(\bold M)|f(x_{\gamma})=0\right \}$ with the norm generated 
by the same inner product.
Then the function $g=f-Pf$ will be the unique solution of the
above minimization problem for the
 functional $g\rightarrow \|\mathcal{L}^{k}g\|_{L_{2}(\bold M)},
 k\geq d$.
 
 Different parts of the following theorem can be found in \cite{Pes08}.

\begin{thm} The following statements hold:

\begin{enumerate}

  \item for any function $f$ from $H^{2k}(\bold M), \>\>\>k\geq d, $ there exists a unique
function $s_{k}(f)$ from the Sobolev space $H^{2k}(\bold M), $ such that
$ f|_{M_{\rho}}=s_{k}(f)|_{M_{\rho}}; $ and this function 
  $s_{k}(f)$ minimizes the functional $u\rightarrow \|\mathcal{L} ^{k}u\|_{L_{2}(\bold M)}$;

\item  every such function $s_{k}(f)$ is of the form 
$$
s_{k}(f)=\sum_{x_{\gamma}\in M_{\rho}}
f(x_{\gamma})L^{2k}_{\gamma}
$$
 where the
function $L^{2k}_{\gamma }\in H^{2k}(\bold M), \>\>\>x_{\gamma }\in M_{\rho}$ minimizes the
same functional and
 takes value $1$ at the point $x_{\gamma}$
and $0$ at all other points of $M_{\rho}$;

\item   functions $L^{2k}_{\gamma}$
form a Riesz basis in the space
 of all polyharmonic functions with singularities on $M_{\rho}$ i.e.  in the
space of such functions from
 $H^{2k}(\bold M )$ which in the sense of distributions satisfy equation

$$\mathcal{L} ^{2k}u=\sum_{x_{\gamma }\in M_{\rho}}\alpha _{\gamma }\delta (x_{\gamma
})$$ where $\delta (x_{\gamma})$ is the Dirac measure at the point
$x_{\gamma }$; 

\item   if in addition the
set $M_{\rho}$ is invariant under some subgroup of diffeomorphisms acting on
$M$ then every two functions $L^{2k}_{\gamma}, L^{2k}_{\mu}$
 are translates of each other.
 \end{enumerate}
 \label{Splines}
 \end{thm}

The crucial role in the proof of the above Theorem \ref{Splines} belongs to the following  lemma which was proved in \cite{Pes00}.

\begin{lem}
 A function $f\in L_{2}(\bold M)$ satisfies equation

$$\mathcal{L}^{2k}f=\sum_{x_{\gamma}\in M_{\rho}}\alpha _{\gamma }\delta (x_{\gamma
}),$$ where
$\{ \alpha _{\gamma} \} \in l_{2}$ if and only if $f$ is a solution to the
minimization problem stated above.
 \end{lem}

Next, if $f\in H^{2k}(\bold M), k\geq d, $ then
 $f-s_{k}(f)\in U^{2k}(M_{\rho})$ and we have for
$k\geq d, $
  $$\|f-s_{k}(f)\|_{L_{2}(\bold M)}\leq
 (C_{0}\rho)^{k}\|\mathcal{L}^{k/2}(f-s_{k}(f))\|_{L_{2}(\bold M)}.$$

Using minimization property of $s_{k}(f)$ we obtain
the inequality 
\begin{equation}
\left\|f-\sum_{x_{\gamma}\in M_{\rho}}f(x{_\gamma})L_{x_{\gamma}}\right\|_{L_{2}(\bold M)}\leq (c_{0}\rho)^{k}\|\mathcal{L} ^{k/2}f\|_{L_{2}(\bold M)}, k\geq d,\label{SobAppr}
\end{equation}
and for $f\in {\mathbf E}_{\omega}(\mathcal{L})$ the Bernstein inequality gives 
for any $f\in {\mathbf E}_{\omega}(\mathcal{L})$ 
\begin{equation}
\left\|f-\sum_{x_{\gamma}\in M_{\rho}} f(x{_\gamma})L_{x_{\gamma}}\right\|_{L_{2}(\bold M)}\leq(c_{0}\rho\sqrt{\omega} )^{k}\|f\|_{L_{2}(\bold M)},\label{PWAppr}
\end{equation}
for $ k\geq d$. The last  inequality shows in particular, that for any $f\in  {\mathbf E}_{\omega}(\mathcal{L})$ one has the following reconstruction algorithm.
\begin{thm}
There exists a $c_{0}=c_{0}(M)$  such that for any $\omega>0$ and any $M_{\rho}$ with 
   $\rho =c_{0}\omega^{-1}$ the following reconstruction formula holds in $L_{2}(M)$-norm
   \begin{equation}
f=\lim_{l\rightarrow \infty}\sum_{x_{j}\in M_{\rho}}f(x_{j})L_{x_{j}}^{(k)}, \  k\geq d,
\end{equation}
for all  $f\in  {\mathbf E}_{\omega}(\mathcal{L})$.
\end{thm}
To develop  a cubature formula   we  introduce the notation 
\begin{equation}
\lambda_{\gamma}^{(k)}=\int_{\bold M}L_{x_{\gamma}}^{(k)}(x)dx,
\end{equation}
where $L_{x_{\gamma}}\in S^{k}(M_{\rho})$ is the Lagrangian spline at the node $x_{\gamma}$.
\begin{thm}

\begin{enumerate}

\item 

For any $f\in H^{2k}(M)$ one has 
 \begin{equation}
\int_{\bold M}f dx\approx\sum_{x_{j}\in M_{\rho}}\lambda_{j}^{(k)}f(x{_j}),\   k\geq d, \label{ApproxInt}
\end{equation}
and the error given by the inequality
\begin{equation}
\left|\int_{\bold M}f dx-\sum_{x_{\gamma}\in M_{\rho}}\lambda_{\gamma}^{(k)}f(x_{\gamma})\right|\leq Vol (\bold M)(c_{0}\rho )^{k}\|\mathcal{L} ^{k/2}f\|_{L_{2}(\bold M)},\label{qubSob}
\end{equation}
for $ k\geq d$.  For a fixed function $f$ the right-hand side of (\ref{qubSob}) goes to zero as long as $\rho$ goes to zero.

\item 
The  formula (\ref{ApproxInt}) is exact for any variational spline $f\in S^{k}(M_{\rho})$ of order $k$ with singularities on $M_{\rho}$.

\end{enumerate}
\end{thm}

By applying the Bernstein inequality we obtain the following theorem.  This result explains our term "asymptotically correct cubature formulas".
\begin{thm}
For any $f\in {\mathbf E}_{\omega}(\mathcal{L})$ one has 
\begin{equation}
\left|\int_{\bold M}f dx-\sum_{x_{\gamma}\in M_{\rho}} \lambda_{\gamma}^{(k)}f(x_{\gamma})\right|\leq Vol(\bold M)(c_{0}\rho\sqrt{\omega} )^{k}\|f\|_{L_{2}(\bold M)},\label{qubPW}
\end{equation}
for $ k\geq d$.  If $\rho=c_{0}\omega^{-1/2}$
the right-hand side in (\ref{qubPW}) goes to zero   for all $f\in {\mathbf E}_{\omega}(\mathcal{L})$ as long as $k$ goes to infinity.
\label{cubPW}
\end{thm}

\section{Positive cubature formulas on compact manifolds}

Let $M_{\rho}=\{x_{k}\}, \ k=1,...,N(M_{\rho}),$ be a $\rho$-lattice on ${\bf M}$. 
We construct the Voronoi partition of   ${\bf M}$ associated to the set 
  $M_{\rho}=\{x_{k}\}, \ k=1,...,N(M_{\rho})$.   Elements of this partition will be denoted as 
$ \mathcal{M}_{k,\rho}$.  Let us recall that the distance from each point in $ \mathcal{M}_{j,\rho}$
to $x_j$ is less than or equal to its distance to any 
other point of the family  $M_{\rho}=\{x_{k}\}, \ k=1,...,N(M_{\rho})$.
Some properties of this cover of   ${\bf M}$ are summarized in the following Lemma. 
which follows easily from  the definitions.

\begin{lem}
\label{mkrho}
The sets $\mathcal{M}_{k,
\rho}, \  k=1,...,N(M_{\rho}),$ have the following properties:

1) they are measurable;

2) they are disjoint;

3) they form a cover of ${\bf M}$;

4) there exist positive $a_{1},\  a_{2}$, independent of $\rho$ and the lattice $M_{\rho}=\{x_{k}\}$, such that 

\begin{equation}
\label{mkrhoway}
a_{1}\rho^{n}\leq \mu\left(\mathcal{M}_{k,\rho}\right)\leq a_{2}\rho^{n}.
\end{equation}

\end{lem}
In what follows we are using partition of unity $\Psi=\{\psi_{\nu}\}$ which appears in (\ref{Sobnorm}).
Our next goal is to prove the following fact.
\begin{thm} Say $\rho > 0$, and let $\left\{\mathcal{M}_{k,\rho}\right\}$ be the disjoint cover
 of ${\bf M}$ which is associated with a $\rho$-lattice $M_{\rho}$.  If $\rho$ is 
sufficiently small  then for any sufficiently large $K\in \mathbb{N}$ there exists a $C(K)>0$ such that for all  smooth functions $f$ the following inequality holds:
\begin{equation}
\left|\sum_{\nu}\sum_{x_{k}\in M_{\rho}}\psi_{\nu}f(x_{k})\  \mu \mathcal{M}_{k,\rho}-\int_{{\bf M}}f(x)dx\right|\leq
$$
$$
 C(K)\sum_{|\beta|=1}^{K}\rho^{n/2+|\beta|}\|(I+\mathcal{L})^{|\beta|/2}f\|_{L_{2}(\bold M)},\label{closeness}
\end{equation}
where $C(K)$ is independent of $\rho$ and the $\rho$-lattice $M_{\rho}$.
\end{thm}
\begin{proof}
 We start with the Taylor series
\begin{equation}
\psi_{\nu}f(y)-\psi_{\nu}f(x_{k})=\sum_{1\leq |\alpha| \leq m-1} \frac{1}{\alpha
!}\partial^{\alpha}(\psi_{\nu}f)(x_{k})(x_{k}-y) ^{\alpha}+\label{Taylor}
\end{equation}
$$
\sum_{|\alpha|=m}\frac{1}{\alpha !}\int_{0}^{\tau}t^{m-1}\partial
^{\alpha}\psi_{\nu}f(x_{k}+t\theta)\theta^{\alpha}dt,
$$
where $ f\in
C^{\infty}(\mathbb{R}^{d}), \   y\in B(x_{k},\rho/2),\  x=(x^{(1)},...,x^{(d)}),\  y=(y^{(1)},...,y^{(d)}),\  \alpha=(
\alpha_{1},...,\alpha_{d}),\>\>(x-y)^{\alpha}=(x^{(1)}-y^{(1)})^{\alpha_{1}}...
(x^{(d)}-y^{(d)})^{\alpha_{d}},\  \tau=\|x-x_{i}\|,\
\theta=(x-x_{i})/ \tau.
$

We are going to use the following inequality, which is essentially the Sobolev imbedding theorem:
\begin{equation}
|(\psi_{\nu}f)(x_{k})|\leq C_{n,m}\sum_{0\leq j \leq m}
\rho^{j-n/p}\|(\psi_{\nu}f)\| _{W^{j}_{p}(B(x_{k},\rho))},\  1\leq p\leq \infty,\>\>
\label{basicineq}
\end{equation}
where $ m>n/p,$ and the functions $\{\psi_{\nu}\}$  form the partition
of unity which we used to define the Sobolev norm in
(\ref{Sobnorm}).
Using (\ref{basicineq}) for $p=1$ we obtain that the following inequality

\begin{equation}\left |\sum_{1\leq|\alpha|\leq m-1} \frac{1}{\alpha
!}\partial^{\alpha}(\psi_{\nu}f)(x_{k})(x_{k}-y) ^{\alpha}\right |\leq 
\end{equation}
$$
C(n,m)\rho^{|\alpha|}\sum_{1\leq|\alpha|\leq m} \sum_{0\leq |\gamma|\leq  m}\rho^{|\gamma|-n}\|\partial
^{\alpha+\gamma}(\psi_{\nu}f)\|_{L_{1}(B(x_{k},\rho))},\ \ m>n,\label{interm}
$$
 for some $C(n,m)\geq 0$. Since, by the Schwarz inequality,
 \begin{equation}
 \|\partial
^{\alpha}(\psi_{\nu}f)\|_{L_{1}(B(x_{k},\rho))}\leq C(n)\rho^{n/2}
 \|\partial
^{\alpha}(\psi_{\nu}f)\|_{L_{2}(B(x_{k},\rho))}
 \end{equation}
 we obtain the following estimate, which holds for small  $\rho$:
 \begin{equation}
\sup_{y\in B(x_{k},\rho)}\left | \sum_{1\leq|\alpha|\leq m-1} \frac{1}{\alpha
!}\partial^{\alpha}(\psi_{\nu}f)(x_{k})(x_{k}-y) ^{\alpha}\right |\leq 
\end{equation}
$$
C(n,m)\sum_{1\leq|\beta|\leq 2m}\rho^{|\beta|-n/2} \|\partial
^{\beta}(\psi_{\nu}f)\|_{L_{2}(B(x_{k},\rho))},\  m>n.
$$
Next, using the Schwarz inequality and the
assumption that
 $m>n=dim\ {\bf M},\  |\alpha|=m,$ we obtain
$$
\left |\int_{0}^{\tau}t^{m-1}\partial
^{\alpha}\psi_{\nu}f(x_{k}+t\theta)\theta^{\alpha}dt\right |\leq
$$
$$\int_{0}^{\tau}t^{m-n/2-1/2}|t^{n/2-1/2}\partial^{\alpha}
\psi_{\nu}f(x_{k}+t\theta)|dt\leq
$$
$$C\left(\int_{0}^{\tau}t^{2m-n-1}\right)^{1/2}\left(\int_{0}
^{\tau} t^{n-1}|\partial^{\alpha}\psi_{\nu}f(x_{k}+t\theta)| ^{2}dt\right)
^{1/2}\leq
$$
$$C\tau^{m-n/2}\left(\int_{0}^{\tau}t^{n-1}|\partial^{\alpha}
\psi_{\nu}f(x_{k}+t\theta)|^{2}dt\right)^{1/2}, \ m>n.
$$
We square this inequality, and integrate both sides of it over the ball
$B(x_{k},\rho/2)$, using the spherical coordinate system
$(\tau, \theta).$  We find

$$
\int_{B(x_{k},\rho)}\left |\int_{0}^{\tau}t^{m-1}\partial
^{\alpha}\psi_{\nu}f(x_{k}+t\theta)\theta^{\alpha}dt\right |^{2}\tau^{n-1}d\theta
d\tau\leq
$$
$$
C(m,n)\int_{0}^{\rho/2}\tau^{2m-n}\int_{0}^{2\pi}
\left |\int_{0}^{\tau}t^{n-1}\partial
^{\alpha}(\psi_{\nu}f)(x_{k}+t\theta)\theta^{\alpha}dt\right |^{2}\tau^{n-1}d\theta
d\tau\leq
$$
$$C(m,n)\int_{0}^{\rho/2}t^{n-1}\left(\int_{0}^{2\pi}\int_{0}^{\rho/2}
\tau^{2m-n}\left |\partial^{\alpha}(\psi_{\nu}f)(x_{k}+t\theta)\right |^{2}
\tau^{n-1}d\tau
d\theta\right)dt\leq
$$
$$
C_{m,n}\rho^{2|\alpha|}\|\partial^{\alpha}
(\psi_{\nu}f)\|^{2}_{L_{2}(B(x_{k},\rho))},
$$
where $\tau=\|x-x_{k}\|\leq\rho/2, \  m=|\alpha|>n.$ 
Let $\left\{\mathcal{M}_{k,\rho}\right\}$ be the Voronoi  cover of ${\bf M}$ which is associated 
with a $\rho$-lattice $M_{\rho}$ (see Lemma \ref{mkrho}). 
From here we obtain
\begin{equation}
\int _{\mathcal{M}_{k}}\left | \psi_{\nu}f(y)-\psi_{\nu}f(x_{k})\right |dx\leq  
\end{equation}
$$
C(n,m)\sum_{1\leq|\beta|\leq 2m}\rho^{|\beta|+n/2} \|\partial
^{\beta}(\psi_{\nu}f)\|_{L_{2}(B(x_{k},\rho))}
+
$$
$$
\sum_{|\alpha|=m}\frac{1}{\alpha !}\int _{B(x_{k}, \rho)}  \left  |\int_{0}^{\tau}t^{m-1}\partial
^{\alpha}\psi_{\nu}f(x_{k}+t\theta)\theta^{\alpha}dt \right |\leq 
$$
$$
C(n,m)\sum_{1\leq|\beta|\leq 2m}\rho^{|\beta|+n/2} \|\partial
^{\beta}(\psi_{\nu}f)\|_{L_{2}(B(x_{k},\rho))}+
$$
$$
\rho^{n/2}
\sum_{|\alpha|=m}\frac{1}{\alpha !}\left (\int _{B(x_{k}, \rho)}  \left  |\int_{0}^{\tau}t^{m-1}\partial
^{\alpha}\psi_{\nu}f(x_{k}+t\theta)\theta^{\alpha}dt \right |^{2}\tau^{n-1}d\tau d\theta\right)^{1/2}\leq 
$$
$$
C(n,m)\sum_{1\leq|\beta|\leq 2m}\rho^{|\beta|+n/2} \|\partial
^{\beta}(\psi_{\nu}f)\|_{L_{2}(B(x_{k},\rho))}.
$$

Next, we have the following inequalities
$$
\sum_{\nu}\sum_{x_{k}\in M_{\rho}}\psi_{\nu}f(x_{k})\  \mu \mathcal{M}_{k,\rho}-\int_{{\bf M}}f(x)dx=
$$
$$
-\sum_{\nu}\left(\sum_{k} \int _{\mathcal{M}_{k,\rho}}\psi_{\nu}f(x)dx-\sum_{k}\psi_{\nu}f(x_{k})\  \mu \mathcal{M}_{k,\rho} \right )\leq
$$
\begin{equation}
\sum_{\nu}\sum_{k}\left | \int _{\mathcal{M}_{k,\rho}}\psi_{\nu}f(x)-\psi_{\nu}f(x_{k})\  \mu \mathcal{M}_{k,\rho}dx \right |
\end{equation}
$$
\leq C(n,m)\rho^{n/2}\sum_{\nu}\sum_{x_{k}\in M_{\rho}}\sum_{1\leq |\beta|\leq 2m}\rho^{|\beta|} \|\partial
^{\beta}(\psi_{\nu}f)\|_{L_{2}(B(x_{k},\rho))},
$$
where $\ m>n.  $
Using the definition of the Sobolev norm and elliptic regularity of the operator $I+\mathcal{L}$, 
where $I$ is the identity operator on $L_{2}({\bf M})$,  we obtain the inequality (\ref{closeness}).
\end{proof}

Now we are going to prove existence of cubature formulas which are exact on $ {\mathbf E}_{\omega}({\bf M})$,
and have positive coefficients of the "right" size.

\begin{thm} 
\label{cubformula}
There exists  a  positive constant $a_{0}$,    such  that if  $\rho=a_{0}(\omega+1)^{-1/2}$, then
for any $\rho$-lattice $M_{\rho}=\{x_{k}\}$, there exist strictly positive coefficients $\mu_{x_{k}}>0, 
 \  x_{k}\in M_{\rho}$, \  for which the following equality holds for all functions in $ {\mathbf E}_{\omega}({\mathcal L})$:
\begin{equation}
\label{cubway}
\int_{{\bf M}}fdx=\sum_{x_{k}\in M_{\rho}}\mu_{x_{k}}f(x_{k}).
\end{equation}
Moreover, there exists constants  $\  c_{1}, \  c_{2}, $  such that  the following inequalities hold:
\begin{equation}
c_{1}\rho^{n}\leq \mu_{x_{k}}\leq c_{2}\rho^{n}, \ n=dim\   {\bf M}.
\end{equation}
\end{thm}
\begin{proof}
By using the Bernstein inequality, and our Plancherel-Polya inequalities  (\ref{completePlPo100}), and assuming that 
\begin{equation}
\rho<\frac{1}{2\sqrt{\omega+1}}
\end{equation}
we obtain from (\ref{closeness}) the following inequality:
\begin{equation}
\label{clos-band}
\left|\sum_{\nu}\sum_{x_{k}\in M_{\rho}}\psi_{\nu}f(x_{k})\  \mu \mathcal{M}_{k,\rho}-\int_{{\bf M}}f(x)dx\right|\leq
C_{1}\rho^{n/2}\sum_{|\beta|=1}^{K}\left(\rho\sqrt{1+\omega }\right)^{|\beta|}\|f\|_{L_{2}(\bold M)}\leq 
$$
$$
C_{2}\rho^{n}\left(\rho\sqrt{1+\omega}\right)\left(\sum_{x_{k}\in M_{\rho}} |f(x_{k})|^{2}\right)^{1/2},
\end{equation}
where $C_{2}$ is independent of  $\rho\in \left(0, (2\sqrt{\omega+1}\right)^{-1})$
and the $\rho$-lattice $M_{\rho}$.

Let $R_{\omega}(\mathcal L)$ denote the space of real-valued functions in 
${\mathbf E}_{\omega}(\mathcal L)$.  Since the eigenfunctions of ${\mathcal L}$ may be taken to be real,
we have ${\mathbf E}_{\omega}(\mathcal L) = R_{\omega}(\mathcal L) + iR_{\omega}(\mathcal L)$, so 
it is enough to show that (\ref{cubway}) holds for all $f \in R_{\omega}(\mathcal L)$.

Consider the sampling operator
$$
S: f\rightarrow \{f(x_{k})\}_{x_{k}\in M_{\rho}},
$$
which maps  $R_{\omega}(\mathcal L)$ into the space  $\mathbb{R}^{|{\mathcal M_{\rho}}|}$
with the $\ell^2$ norm.
Let $V = S(R_{\omega}(\mathcal L))$ be the image of  $R_{\omega}(\mathcal L)$ under $S$.
$V$ is a  subspace of $\mathbb{R}^{|{\mathcal M_{\rho}}|}$,
and we consider it with the induced $\ell^2$ norm.
If $u \in V$, denote the linear functional $y \to (y,u)$ on $V$ by $\ell_u$.
By our  Plancherel-Polya inequalities (\ref{completePlPo100})
,  the map 
$$
\{f(x_k)\}_{x_{k}\in M_{\rho}} \to \int _{\bf M}fdx
$$ 
is a well-defined linear functional on the finite dimensional space
$V$, and so equals $\ell_v$ for some $v \in V$, which may
or may not  have all components positive.  On the other hand, if $w$ is the vector with components $\{\mu({\mathcal M}_{k,\rho})\}, \ x_{k}\in M_{\rho}$, then $w$ might
not be in $V$, but it has all components positive and of the right size
$$
a_{1}\rho^{n}\leq \mu\left(\mathcal{M}_{k,\rho}\right)\leq a_{2}\rho^{n},
$$
for some  positive $a_{1},\  a_{2}$, independent  of $\rho$ and the lattice $M_{\rho}=\{x_{k}\}$.
Since, for any vector $u \in V$ the norm of $u$  is exactly the norm of the corresponding functional 
$\ell_u$,  inequality (\ref{clos-band})
tells us that 
\begin{equation}
\label{2}
\|Pw-v\| \leq \|w-v\| \leq C_2\rho^n \left(\rho\sqrt{1+\omega}\right),
\end{equation}
where  $P$ is the orthogonal projection onto $V$. Accordingly, if $z$ is the real vector $v-Pw$, then
\begin{equation}
\label{3}
v+(I-P)w = w + z ,
\end{equation}
where $\|z\| \leq C_2\rho^n \left(\rho\sqrt{1+\omega}\right)$.  Note, that all components of the vector $w$ 
 are of order $O(\rho^{n})$, while the order of $\|z\|$ is  $O(\rho^{n+1})$. Accordingly, if 
 $\rho\sqrt{1+\omega}$ is sufficiently small, then 
$\mu := w + z$ has all components positive and of the right size.  Since $\mu = v + (I-P)w$, the linear
functional $y \to (y,\mu)$ on $V$ equals $\ell_v$.  In other words, if the vector $\mu$ has components
 $\{\mu_{x_{k}}\}, \ x_{k}\in M_{\rho},$ then 
$$
\sum_{x_{k}\in M_{\rho}}f(x_{k})\mu_{x_{k}} = \int_{\bf M} f dx
$$ 
for all $f \in R_{\omega}(\mathcal L)$, and hence for all $f \in {\mathbf E}_{\omega}(\mathcal L)$, as desired.

\end{proof}

We obviously have the following result.

\begin{thm}

\begin{enumerate}

\item  There exists a $c_{0}=c_{0}(\bold M,\mathcal L)$  and for
   any $0\leq k\leq m, k,m\in \mathbb{N},$ there exists a constant $C_{k,m}>0$ such
   that if $M_{\rho}=\{x_{k}\}$ is a $\rho$-lattice  with
   $0<\rho\leq c_{0}\omega^{-1}$ then for the same weights $\{\mu_{x_{j}}\}$ as in (\ref{cubway})
\begin{equation}
\left |\int_{\bold M}f-
\sum_{x_{j}}f_{\omega}(x_{j})\mu_{x_{j}}\right |\leq\frac{C_{
k, m}}{\omega^{k}}\Omega_{m-k}\left(\mathcal{L}^{k}f, 1/\omega\right),
\end{equation}

\item For functions in $
\mathbf{B}_{2,\infty}^{\alpha}(\bold M)$ the following relation holds
\begin{equation}
\left |\int_{\bold M}f-\sum_{x_{j}}f_{\omega}(x_{j})\mu_{x_{j}}\right |=O(\omega^{-\alpha}),\>\>\>\omega\longrightarrow
\infty.
\end{equation}
where $f_{\omega}$ is the orthogonal projection of $f\in L_{2}(\bold M)$
onto ${\mathbf E}_{\omega}(\mathcal{L})$.

\end{enumerate}

\end{thm}

\section{Harmonic Analysis on Compact homogeneous manifolds}

We review some very basic notions of harmonic analysis on
compact homogeneous manifolds \cite{H3}, Ch. II.

 Let ${\bf M},\ dim \>{\bf M}=n,$ be a
compact connected $C^{\infty}$-manifold. One says  that a compact
Lie group $G$ effectively acts on ${\bf M}$ as a group of
diffeomorphisms if:

1)  every element $g\in G$ can be identified with a diffeomorphism
$$
g: {\bf M}\rightarrow {\bf M}
$$
of ${\bf M}$ onto itself and
$$
g_{1}g_{2}\cdot x=g_{1}\cdot(g_{2}\cdot x),\ g_{1}, g_{2}\in G,\
x\in {\bf M},
$$
where $g_{1}g_{2}$ is the product in $G$ and $g\cdot x$ is the
image of $x$ under $g$,

2) the identity $e\in G$ corresponds to the trivial diffeomorphism
\begin{equation}
e\cdot x=x,
\end{equation}

3) for every $g\in G,\ g\neq e,$ there exists a point $x\in {\bf M}$ such
that $g\cdot x\neq x$.

\bigskip

A group $G$ acts on ${\bf M}$ \textit{transitively} if in addition to
1)- 3) the following property holds:

4) for any two points $x,y\in {\bf M}$ there exists a diffeomorphism
$g\in G$ such that
$$
g\cdot x=y.
$$

A \textit{homogeneous} compact manifold ${\bf M}$ is a
$C^{\infty}$-compact manifold on which a compact
Lie group $G$ acts transitively. In this case ${\bf M}$ is necessarily of the form $G/K$,
where $K$ is a closed subgroup of $G$. The notation $L_{2}({\bf M}),
$ is used for the usual Banach  spaces
$L_{2}({\bf M},dx)$, where $dx$ is an invariant
measure.

Every element $X$ of the (real) Lie algebra of $G$ generates a vector
field on ${\bf M}$, which we will denote by the same letter $X$. Namely,
for a smooth function $f$ on ${\bf M}$ one has
$$
 Xf(x)=\lim_{t\rightarrow 0}\frac{f(\exp tX \cdot x)-f(x)}{t}
 $$
for every $x\in {\bf M}$. In the future we will consider on ${\bf M}$ only
such vector fields. The translations along integral curves of such
vector fields $X$ on ${\bf M}$  can be identified with a one-parameter
group of diffeomorphisms of ${\bf M}$, which is usually denoted as $\exp
tX, -\infty<t<\infty$. At the same time, the one-parameter group
$\exp tX, -\infty<t<\infty,$ can be treated as a strongly
continuous one-parameter group of operators acting on the space $L_{2}({\bf M})$.  These operators act on functions according to the
formula
$$
f\rightarrow f(\exp tX\cdot x), \        t\in \mathbb{R}, \
  f\in L_{2}({\bf M}),\        x\in {\bf M}.
$$
 The
generator of this one-parameter group will be denoted by $D_{X}$,
and the group itself will be denoted by
$$
e^{tD_{X}}f(x)=f(\exp tX\cdot x),\       t\in \mathbb{R}, \
     f\in L_{2}({\bf M}), \       x\in {\bf M}.
$$

According to the general theory of one-parameter groups in Banach
spaces,  the operator $D_{X}$ is a closed
operator on every $L_{2}({\bf M})$.  

If $\textbf{g}$ is the Lie algebra of a compact Lie group $G$ then
(\cite{H3}, Ch.\ II,) it is a direct sum
$\textbf{g}=\textbf{a}+[\textbf{g},\textbf{g}]$, where
$\textbf{a}$ is the center of $\textbf{g}$, and
$[\textbf{g},\textbf{g}]$ is a semi-simple algebra. Let $Q$ be a
positive-definite quadratic form on $\textbf{g}$ which, on
$[\textbf{g},\textbf{g}]$, is opposite to the Killing form. Let
$X_{1},...,X_{d}$ be a basis of
$\textbf{g}$, which is orthonormal with respect to $Q$.
 Since the form $Q$ is $Ad(G)$-invariant, the operator
$$
-X_{1}^{2}-X_{2}^{2}-\    ... -X_{d}^{2},    \ d=dim\ G
$$
is a bi-invariant operator on $G$. This implies in particular that
the
   corresponding operator on $L_{2}({\bf M})$
\begin{equation}
\mathcal{L}=-D_{1}^{2}- D_{2}^{2}- ...- D_{d}^{2}, \>\>\>
       D_{j}=D_{X_{j}}, \        d=dim \ G,\label{Laplacian}
\end{equation}
\textit{commutes} with all operators $D_{j}=D_{X_{j}}$. This operator
$\mathcal{L}$, which is usually called the Laplace operator, is elliptic, and is
involved in most of the constructions and results of our paper.

In the rest of the paper, the notation $\mathbb{D}=\{D_{1},...,
D_{d}\},\>\>\> d=dim \ G,$ will be used for the differential operators
on $L_{2}({\bf M}),$ which are involved in the
formula (\ref{Laplacian}).

There are 
situations in which the operator $\mathcal{L}$ is, or is proportional to, the
Laplace-Beltrami operator  of an invariant metric on ${\bf M}$. This
happens for example, if ${\bf M}$ is a $n$-dimensional torus, a compact semi-simple
Lie group, or a compact symmetric space of rank one.

\section{On the product of eigenfunctions of the Casimir operator $\mathcal{L}$ on compact
 homogeneous manifolds}

In this section, we will use the assumption that ${\bf M}$ is a compact homogeneous manifold,
and that ${\mathcal L}$ is the operator of (\ref{Laplacian}), in an essential way.

\begin{thm}
\label{prodthm}
If ${\bf M}=G/K$ is a compact homogeneous manifold and $\mathcal{L}$
is defined as in (\ref{Laplacian}), then for any $f$ and $g$ belonging
to ${\mathbf E}_{\omega}(\mathcal{L})$,  their product $fg$ belongs to
${\mathbf E}_{4d\omega}(\mathcal{L})$, where $d$ is the dimension of the
group $G$.

\end{thm}

\begin{proof} First, we show that if for an $f\in L_{2}({\bf M})$ and a positive $\omega$ there exists a constant $C(f,\omega)$ such that the following inequalities hold 
\begin{equation}
\|\mathcal{L}^{k}f\|_{L_{2}(\bold M)}\leq C(f,\omega)\omega^{k}\|f\|_{L_{2}(\bold M)}
\end{equation}
for all natural $k$ then $f\in  {\mathbf E}_{\omega}(\mathcal{L})$.
Indeed, assume that
$$
\lambda_{m}\leq \omega<\lambda_{m+1}
$$ 
and
\begin{equation}
f=\sum_{j=0}^{\infty}c_{j}u_{j}\label{Fseries},
\end{equation}
$$
c_{j}(f)=<f,u_{j}>=\int_{{\bf M}}f(x)\overline{u_{j}(x)}dx.
$$
Then by the Plancherel Theorem
$$
\lambda_{m+1}^{2k}\sum_{j=m+1}^{\infty}|c_{j}|^{2}\leq
\sum_{j=m+1}^{\infty}|\lambda_{j}^{k}c_{j}|^{2}\leq
\|\mathcal{L}^{k}f\|_{L_{2}(\bold M)}^{2}\leq
C^{2}\omega^{2k}\|f\|_{L_{2}(\bold M)}^{2},\>\>\>C=C(f,\omega),
$$
which implies
$$
\sum_{j=m+1}^{\infty}|c_{j}|^{2}\leq C^{2}
\left(\frac{\omega}{\lambda_{m+1}} \right)^{2k}\|f\|_{L_{2}(\bold M)}^{2}.
$$
In the last inequality the fraction $\omega/\lambda_{m+1}$ is
strictly less than $1$  and $k$ can be any natural number. This
shows that the series (\ref{Fseries}) does not contain terms with
$j\geq m+1$, i.e.\ the function $f$ belongs to $
\textbf{E}_{\omega}(\mathcal{L})$.

Now, since every smooth vector field on ${\bf M}$ is a differentiation
of the algebra $C^{\infty}({\bf M})$, one has that for every operator
$D_{j}, 1\leq j\leq d,$ the following equality holds for any two
smooth functions $f$ and $g$ on ${\bf M}$:
\begin{equation}
D_{j}(fg)=fD_{j}g+gD_{j}f, \>\>\> 1\leq j\leq d.
\end{equation}
Using formula (\ref{Laplacian}) one can easily verify that for any
natural $k\in \mathbb{N}$, the term
$\mathcal{L}^{k}\left(fg\right)$ is a sum of $\>\>d^{k},
\>\>\>(d=dim G),$ terms of the following form:
\begin{equation}
D_{j_{1}}^{2}...D_{j_{k}}^{2}(fg),\ 1\leq j_{1},...,j_{k}\leq d.
\end{equation}
For every $D_{j}$ one has
$$
D_{j}^{2}(fg)=f(D_{j}^{2}g)+2(D_{j}f)(D_{j}g)+g(D_{j}^{2}f).
$$
Thus, the function $\mathcal{L}^{k}\left(fg\right)$ is a sum of
$(4d)^{k}$ terms of the form
$$
(D_{i_{1}}...D_{i_{m}}f)(D_{j_{1}}...D_{j_{2k-m}}g).
$$
This implies that
\begin{equation}
\left|\mathcal{L}^{k}\left(fg\right)\right|\leq
(4d)^{k}\sup_{0\leq m\leq 2k}\sup_{x,y\in
{\bf M}}\left|D_{i_{1}}...D_{i_{m}}f(x)\right|\left|D_{j_{1}}...D_{j_{2k-m}}g(y)\right|.\label{estim3}
\end{equation}
Let us show that the following inequalities hold:
\begin{equation}
\|D_{i_{1}}...D_{i_{m}}f\|_{L_{2}(\bold M)}\leq
\omega^{m/2}\|f\|_{L_{2}(\bold M)}\label{estim1}
\end{equation}
and
\begin{equation}
\|D_{j_{1}}...D_{j_{2k-m}}g\|_{L_{2}(\bold M)}\leq
\omega^{(2k-m)/2}\|g\|_{L_{2}(\bold M)}\label{estim2}
\end{equation}
for all $f,g \in {\mathbf E}_{\omega}(\mathcal{L})$.  First, we note  that the operator
$$
-\mathcal{L}=D_{1}^{2}+...+D_{d}^{2}
$$
commutes with every $D_{j}$ (see the explanation before the
formula (\ref{Laplacian}) ).
 The same is
true for $\mathcal{L}^{1/2}$. But then
$$
\|\mathcal{L}^{1/2}f\|_{L_{2}(\bold M)}^{2}=<\mathcal{L}^{1/2}f,\mathcal{L}^{1/2}f>=<\mathcal{L}
f,f>=
$$
$$
-\sum_{j=1}^{d}<D_{j}^{2}f,f>=\sum_{j=1}^{d}<D_{j}f,D_{j}f>=
\sum_{j=1}^{d}\|D_{j}f\|_{L_{2}(\bold M)}^{2},
$$
and also
$$
\|\mathcal{L}f\|_{L_{2}(\bold M)}^{2}=\|\mathcal{L}^{1/2}\mathcal{L}^{1/2}f\|_{L_{2}(\bold M)}^{2}=
\sum_{j=1}^{d}\|D_{j}\mathcal{L}^{1/2}f\|_{L_{2}(\bold M)}^{2}=
$$
$$
\sum_{j=1}^{d}\|\mathcal{L}^{1/2}D_{j}f\|_{L_{2}(\bold M)}^{2}=\sum_{j,k=1}^{d}\|D_{j}D_{k}f\|_{L_{2}(\bold M)}^{2}.
$$
From here by induction on $s\in \mathbb{N}$ one can  obtain the
following equality:
\begin{equation}
\|\mathcal{L}^{s/2}f\|_{L_{2}(\bold M)}^{2}=\sum_{1\leq i_{1},...,i_{s}\leq
d}\|D_{i_{1}}...D_{i_{s}}f\|_{L_{2}(\bold M)}^{2},\ s\in \mathbb{N}, \label{eq0}
\end{equation}
which implies the estimates (\ref{estim1}) and (\ref{estim2}). For
example, to get (\ref{estim1}) we take a function $f$ from
${\mathbf E}_{\omega}(\mathcal{L})$, an $m\in \mathbb{N}$ and  do the
following
\begin{equation}
\|D_{i_{1}}...D_{i_{m}}f\|_{L_{2}(\bold M)}\leq \left(\sum_{1\leq
i_{1},...,i_{m}\leq
d}\|D_{i_{1}}...D_{i_{m}}f\|_{L_{2}(\bold M)}^{2}\right)^{1/2}=
$$
$$
\|\mathcal{L}^{m/2}f\|_{L_{2}(\bold M)}\leq
\omega^{m/2}\|f\|_{L_{2}(\bold M)}.\label{estim4}
\end{equation}
 In a similar way we obtain (\ref{estim2}).

The formula
(\ref{estim3}) along with the formula (\ref{estim4}) imply the
estimate
\begin{equation}
\|\mathcal{L}^{k}(fg)\|_{L_{2}(\bold M)}\leq (4d)^{k}\sup_{0\leq m\leq
2k}\|D_{i_{1}}...D_{i_{m}}f\|_{L_{2}(\bold M)}\|D_{j_{1}}...D_{j_{2k-m}}g\|_{\infty}\leq
$$
$$(4d)^{k}\omega^{m/2}\|f\|_{L_{2}(\bold M)}\sup_{0\leq m\leq
2k}\|D_{j_{1}}...D_{j_{2k-m}}g\|_{\infty}.
\end{equation}
Using the Sobolev embedding Theorem and elliptic regularity of
$\mathcal{L}$, we obtain for every $s>\frac{dim {\bf M}}{2}$
\begin{equation}
\|D_{j_{1}}...D_{j_{2k-m}}g\|_{\infty}\leq
C({\bf M})\|D_{j_{1}}...D_{j_{2k-m}}g\|_{H^{s}({\bf M})}\leq
$$
$$
C({\bf M})\left\{\|D_{j_{1}}...D_{j_{2k-m}}g\|_{L_{2}(\bold M)}+
\|\mathcal{L}^{s/2}D_{j_{1}}...D_{j_{2k-m}}g\|_{L_{2}(\bold M)}\right\},
\end{equation}
where $H^{s}({\bf M})$ is the Sobolev space of $s$-regular functions on
${\bf M}$. Since the operator $\mathcal{L}$ commutes with each of the
operators $D_{j}$, the estimate (\ref{estim4}) gives the following
inequality:
\begin{equation}
\|D_{j_{1}}...D_{j_{2k-m}}g\|_{\infty}\leq
C({\bf M})\left\{\omega^{k-m/2}\|g\|_{L_{2}(\bold M)}+\omega^{k-m/2+s}\|g\|_{L_{2}(\bold M)}\right\}\leq
$$
$$
C({\bf M})\omega^{k-m/2}\left\{\|g\|_{L_{2}(\bold M)}+\omega^{s/2}\|g\|_{L_{2}(\bold M)}\right\}=
C({\bf M},g,\omega,s)\omega^{k-m/2},\>\>\>s>\frac{dim\  {\bf M}}{2}.
\end{equation}
Finally we have the following estimate:
\begin{equation}
\|\mathcal{L}^{k}(fg)\|_{L_{2}(\bold M)}\leq
C({\bf M},f,g,\omega,s)(4d\omega)^{k},\>\>\>s>\frac{dim \ {\bf M}}{2},\>\>k\in
\mathbb{N},
\end{equation}
which leads to our result.
The Theorem
is proved.
\end{proof}

\end{document}